\documentclass[12pt, reqno]{amsart}

\usepackage{amssymb, amsmath, amsthm, enumerate, mathtools, graphicx, tikz, color, soul, hyperref, bbm, nameref}
\usepackage{stmaryrd}
\usepackage[margin=1 in]{geometry}
\usepackage[mathscr]{euscript}
\allowdisplaybreaks

\usepackage[dvipsnames]{xcolor}

\usepackage{parskip}

\usepackage{pgfplots}

\xdefinecolor{darkgreen}{RGB}{0, 153, 0}

\usepackage{amsmath}
\usepackage{verbatim}

\DeclareFontFamily{U}{mathx}{\hyphenchar\font45}
\DeclareFontShape{U}{mathx}{m}{n}{
      <5> <6> <7> <8> <9> <10>
      <10.95> <12> <14.4> <17.28> <20.74> <24.88>
      mathx10
      }{}
\DeclareSymbolFont{mathx}{U}{mathx}{m}{n}
\DeclareFontSubstitution{U}{mathx}{m}{n}
\DeclareMathAccent{\widecheck}{0}{mathx}{"71}
\DeclareMathAccent{\wideparen}{0}{mathx}{"75}

\usepackage{colonequals}

\newcommand{\C}{\mathbb{C}}
\newcommand{\R}{\mathbb{R}}

\newcommand{\Z}{\mathbb{Z}}

\newcommand{\N}{\mathbb{N}}

\newcommand{\F}{\mathbb{F}}

\DeclareMathOperator{\mult}{mult}

\def\XXint#1#2#3{{\setbox0=\hbox{$#1{#2#3}{\int}$}
     \vcenter{\hbox{$#2#3$}}\kern-.5\wd0}}

\newtheorem{theorem}{Theorem}[section]

\newtheorem{lemma}[theorem]{Lemma}

\newtheorem{corollary}[theorem]{Corollary}
\theoremstyle{definition}
\newtheorem{definition}[theorem]{Definition}

\newtheorem{remark}[theorem]{Remark}

\numberwithin{equation}{section}
\usepackage{enumerate}

\usepackage{marginnote}

\begin{document}
{\allowdisplaybreaks 

\renewcommand{\labelenumii}{\arabic{enumi}.\arabic{enumii}}

\title{On an analogue of BRK-type sets in finite fields}
\author{Madeline Forbes}
\date{July 8, 2026}

\newcommand{\Addresses}{{
\bigskip
\footnotesize

{\textsc{Department of Mathematics, 1984 Mathematics Road, University of British Columbia, Vancouver, BC V6T 1Z2, Canada}} \par\nopagebreak
 \textit{E-mail address}: \texttt{mforbes3@student.ubc.ca}
}}

\maketitle

\begin{abstract}
A Besicovitch-Rado-Kinney (BRK) set in $\R^n$ contains a hypersphere of every radius. In $\F_q^n$, BRK-type sets of degree $\ell$ analogously contain a family of $(n-1)$-dimensional surfaces, parametrized by a dilation factor and determined by a fixed homogeneous polynomial of degree $\ell$.
We define $(n,d)$-BRK-type sets of degree $\ell$, 
which contain a family of $d$-dimensional sets parametrized by an $(n-d)$-dimensional dilation factor and determined by fixed homogeneous polynomials of degree $\ell$. 
We use the polynomial method to obtain a lower bound $|S| \gtrsim_{n, \ell} q^n$ on $(n,d)$-BRK-type sets $S$ of degree $\ell$. 
We obtain an improved lower bound $|S| \geq \frac{(q-1)^n}{(\ell + 1 - 2\ell/q)^n}$ by implementing the method of multiplicities; 
this is the same bound obtained by Trainor on BRK-type sets of degree $\ell$, and we obtain this bound independently of $d$.
\end{abstract}

\section{Introduction}\label{Introduction}
Let $\F_q$ be a finite field with $q$ elements, and let $n \geq 2$. 
We say that $K \subset \F_q^n$ is a Kakeya set if it contains a line in every direction: 
for every nonzero $a \in \F_q^n$, 
there exists $b \in \F_q^n$ with \[\{ax + b: x \in \F_q\} \subset K.\] 
As an analogue to the famous Euclidean Kakeya problem, 
Wolff \cite{Wolff2} introduced the finite field Kakeya problem, 
which seeks a lower bound on the cardinality of Kakeya sets in $\F_q^n$. 
Given that Kakeya sets are ``large'' in the sense that they contain many lines,
he conjectured that Kakeya sets in $\F_q^n$ must also have a large cardinality: 
specifically, 
he conjectured that such sets should have cardinality at least $c_nq^n$, 
where $c_n$ is a constant depending only on $n$. 
This lower bound was first obtained by
Dvir \cite{Dvir} using the polynomial method. 
Dvir, Kopparty, Saraf and Sudan \cite{DKSS} later used an extension of the polynomial method, 
the method of multiplicities, 
to obtain a near-optimal lower bound:
it is possible to construct Kakeya sets of only slightly larger cardinality than their lower bound \cite{Saraf-Sudan}. As in the Euclidean setting, these objects are ‘large’ in a dimensional sense, but our focus is on finite-field analogues.

A natural way to extend the idea of sets in $\F_q^n$ containing many lines is to consider sets in $\F_q^n$ containing many curves, and to see if they must also have a large cardinality. This has been studied by Ellenberg, Oberlin and Tao \cite{EOT}, who obtained a generalization of Dvir's result by showing existence of lower bounds on sets in $\F_q^n$ containing many irreducible algebraic curves. Warren and Winterhof \cite{WW} obtained an explicit bound $\bigl(\frac{q-1}{2n}\bigr)^n$ in the specific case of conical Kakeya sets: sets containing many hyperbolas or parabolas. 

Another variant on Kakeya sets are Besicovitch-Rado-Kinney (BRK) sets, which in two dimensions contain many circles: specifically, BRK sets contain a circle of every radius. In higher dimensions, this is generalized to sets that contain a hypersphere of every radius. Such sets were first considered in the Euclidean case by Besicovitch-Rado \cite{Besicovitch-Rado} and Kinney \cite{Kinney}, who each constructed BRK sets of Lebesgue measure $0$. As an analogue to the Euclidean Kakeya problem, Kolasa and Wolff \cite{Kolasa-Wolff} and Wolff \cite{Wolff1} considered BRK sets in $\R^n$; they showed that such sets are large in the sense that they must have Hausdorff dimension $n$. 

BRK sets have been also studied over $\F_q^n$. In \cite{MWW}, Makhul, Warren and Winterhof obtained lower bounds of order $q^n$ on BRK sets in $\F_q^n$, where they take a sphere of radius $r$ in $\F_q^n$ to be a subset of the form 
\[\{x \in \F_q^n: (x_1 - a_1)^2 + \cdots + (x_n - a_n)^2 = r\}.\] 
This is similar to the definition of a hypersphere of radius $r$ in $\R^n$. Observe that the highest homogeneous part of the polynomial determining each such sphere is always $\sum_{i=1}^n x_i^2$; the lower order terms in the polynomial correspond to a translation of the centre of the sphere. It is thus natural to extend this construction to sets that contain many hypersurfaces determined by polynomials of degree $\ell$, where the highest homogeneous part of the polynomials are identical (up to dilation) and lower degree terms may vary. Multiplying each polynomial by a dilation factor then corresponds to choosing the radius of each sphere. Trainor \cite{Trainor} used this to generalize the notion of BRK sets in $\F_q^n$ by defining BRK-type sets of degree $\ell$.

\begin{definition} \cite[Definition 1.1]{Trainor} \label{BRK-type-1}
    If $f \in \F_q [s_1, . . . , s_{n-1}]$ is a homogeneous polynomial, let $P_f \subset
\F_q [s_1, . . . , s_{n-1}]$ be the set of all polynomials with homogeneous part of highest degree equal to $f$. \\
Let $\ell \geq 2$. We say that $S \subset \F_q^n$
is a BRK-type set of degree $\ell$ if there exists a polynomial
$g \in \F_q [s_1, . . . , s_{n-1}]$, with $\deg g = \ell$, so that the following holds: for any $\rho\in \F_q$ , there exist $a = a(\rho) \in \F_q^n$
and $g_\rho \in P_g$ such that 
\[\Lambda_\rho := \{a + \rho(\lambda, g_\rho(\lambda)) :\lambda \in \F_q^{n-1} \} \subset S.\]
\end{definition}
\bigbreak

As an example of such a set, consider the case where $\ell = 2$ and $g(x) = x_1^2 + \cdots + x_{n-1}^2$. With this choice of $g$, up to an added constant, each $g_\rho(x)$ can be written in the form $(x_1 - b_1)^2 + \cdots + (x_{n-1} - b_{n-1})^2$ for some $\{b_i\}_{i=1}^n \subset \F_q$. Up to translation, our set then contains every possible dilation of an $(n-1)$-dimensional paraboloid.

Trainor \cite{Trainor} used the method of multiplicities to establish a lower bound on such sets of order $(q/\ell)^n$ for $n \geq 2$, and using inclusion-exclusion strengthened this to a bound of order $q^n/(2\ell)$ for $n \geq 3$. Given that variants on Kakeya sets in the literature tend to be defined as either one-dimensional families of $(n-1)$-dimensional sets or $(n-1)$-dimensional families of one-dimensional sets, Trainor also proposed ``bridging this gap'' by defining BRK-type sets containing $(n-d)$-dimensional families of $d$-dimensional sets determined by polynomials, where $1 \leq d \leq n-1$. We formulate this proposed definition in a precise manner below.

\begin{definition} \label{BRK-type-general} Let $S \subset \F_q^n$ and $\ell \geq 2$. Suppose that there exist homogeneous polynomials $h_1, \dots, h_{n-d} \in \F_q[x_1, \dots, x_{d}]$ of degree $\ell$, so that for every $\rho = (\rho_1, \dots, \rho_{n-d}) \in \F_q^{n-d}$, there exist $a_\rho \in \F_q^n$ and $\{g_{\rho,i}\}_{i=1}^{n-d} \subset \F_q[x_1, \dots, x_{d}]$, with each $g_{\rho, i} \in \mathcal{P}_{h_i}$, such that
\[\Lambda_\rho := \{\Gamma_\rho(t) := a_\rho + (t, \rho_1 g_{\rho,1}(t), \dots, \rho_{n-d} g_{\rho,n-d}(t)) :  t \in \F_q^d\} \subset S.\]
Then we say that $S$ is an $(n,d)$-BRK-type set of degree $\ell$.
\end{definition}

An $(n,d)$-BRK-type set of degree $\ell$ encodes an $(n-d)$-parameter family of $d$-dimensional sets. These sets are parametrized by a $(n-d)$-dimensional dilation factor, and  are determined by polynomials which have a common highest homogeneous part of degree $\ell$ in each coordinate. In particular, $(n, 1)$-BRK-type sets contain an $(n-1)$-dimensional family of algebraic curves, and $(n, n-1)$-BRK-type sets contain a one-dimensional family of hypersurfaces. 

We briefly discuss two simple examples to illustrate both cases. First, $(n, 1)$-BRK-type sets of degree $1$ contain lines in each direction $a \in \{(1, \dots, 1, \rho_1, \dots, \rho_{n-d}): \rho \in \F_q^d\} \subset \F_q^n$. This is similar to the definition of a Kakeya set, but we require only a subset of line directions to appear within the set. Any Kakeya set in $\F_q^n$, then, must contain an $(n,1)$-BRK-type set of degree $1$. Second, after a reparametrization, $(n, n-1)$-BRK-type sets of degree $\ell=2$ coincide with BRK-type sets of degree $\ell=2$ as in Definition 1.1. Then the set described below Definition 1.1, which contains every possible dilation of an $(n-1)$-dimensional paraboloid, is an example of an $(n, n-1)$-BRK-type set of degree $2$. Note that Definition 1.1 is generally distinct from Definition 1.2 in the case $(n, n-1)$: in Definition 1.2, the coordinates that parametrize the subset are not multiplied by the dilation factor. The definitions are thus the same, up to reparametrization, only in the case $\ell = 2$.

When $1 < d < n-1$, one may interpret the constituent subsets geometrically by considering their projections onto subsets of their coordinates. As an example, we consider the case where each $g_{i, \rho}(t)$ is equal to $h_i(t)$, with $d = \ell =2$, $n=4$, $h_1(t) = t_1^2$, and $h_2(t) = t_2^2$. Defining $\pi_{ij}(x) = (x_i, x_j)$, by construction $\pi_{13}(S)$ and $\pi_{24}(S)$ each contain a parabola of every aperture. 

Our main result is the following:

\begin{theorem}\label{pm-bound}
    Let $S \subset \F_q^n$. Suppose that $S$ is a $(n, d)$-BRK-type set of degree $\ell$, where $\ell \geq 2$. Then $|S| \geq \binom{\lfloor\frac{q-1}{\ell}\rfloor + n}{n}$.
\end{theorem}

This establishes a lower bound $|S| \gtrsim_{n, \ell}q^n$ on $(n,d)$-BRK-type sets of degree $\ell$, for $\ell \geq 2$. Such a bound is best possible, up to improvement of the associated constant. Interestingly, this bound does not depend on $d$, the dimension of the constituent subsets. Our proof uses the polynomial method: if $|S| < \binom{\lfloor\frac{q-1}{\ell}\rfloor + n}{n}$, we can construct a nonzero polynomial in $\rho$ which is shown to have too many zeros. It is possible to slightly modify our argument to obtain the $\ell=2$ bound for the $\ell=1$ case, but we omit this as the $\ell=1$ case is addressed in our improved bound below.

 We also implement the method of multiplicities as inspired by \cite{Trainor} to obtain an improved lower bound. In particular, we obtain the same lower bound on $(n,d)$-BRK-type sets as Trainor obtained on her BRK-type sets, independently of $d$.

\begin{theorem}\label{mm-bound}
    Let $S \subset \F_q^n$. Suppose that $S$ is an $(n, d)$-BRK-type set of degree $\ell$, where $\ell \geq 1$. Then $|S| \geq \frac{(q-1)^n}{(\ell + 1 - 2\ell/q)^n}$.
\end{theorem}

The proof of Theorem \ref{mm-bound} is an extension of our proof for Theorem \ref{pm-bound}: if $|S| < \frac{(q-1)^n}{(\ell + 1 - 2\ell/q)^n}$, we can construct an analogous nonzero polynomial in $\rho$ that vanishes on all $\rho \in (\F_q \setminus \{0\})^{n-d}$ with high multiplicity. This polynomial is shown to have too many zeros, counting multiplicities.

In Section \ref{preliminaries}, we set notation, present our main tools (the polynomial method, Hasse derivatives and multiplicities, and the graded lexicographic order), and give a few preliminary results. We then prove Theorem \ref{pm-bound} via the polynomial method in Section \ref{pm-proof}. In Section \ref{mm-proof}, we upgrade this to Theorem \ref{mm-bound} using the method of multiplicities.

\section{Preliminaries} \label{preliminaries} 

\subsection{Notation}\label{notation}

We briefly introduce some notation which we will use throughout. We fix $n$ and $d$ as in Definition \ref{BRK-type-general}, and use $m$ in lemmas which will later be applied to both cases $m=n$ and $m=d$. For $x = (x_1, \dots, x_m) \in \F_q^m$ and $\alpha = (\alpha_1, \dots, \alpha_m) \in \Z_{\geq 0}^m$, where $m \in \N$, we write 
\[x^\alpha = \prod_{i=1}^m x_i^{\alpha_i}.\]
If we also have $\beta = (\beta_1, \dots, \beta_m)$, the multi-index binomial coefficients are defined by 
\[\binom{\alpha}{\beta} = \prod_{i=1}^n \binom{\alpha_i}{\beta_i},\]
with the convention that $\binom{\alpha}{\beta} = 0$ if $\alpha_i < \beta_i$ for any $i$.
We say that the magnitude of the exponent $\alpha$ is $|\alpha| = \sum_{i=1}^m \alpha_i$. For $\alpha = (\alpha_1, \dots, \alpha_n) \in \Z_{\geq 0}^n$, we write \[\alpha' = (\alpha_{d+1}, \dots, \alpha_n)\] to denote taking only the final $n-d$ coordinates of $\alpha$.

Let $f(x)= \sum_{\alpha \in \Z_{\geq 0}^m} c_\alpha x^\alpha \in \F_q[x_1, \dots, x_m]$. We say that the degree of $f$ is 
\[\deg f := \max\{|\alpha|: c_\alpha \neq 0\}.\]
We say that $f$ is nonzero if there exists $\alpha$ so that $c_\alpha \neq 0$; this is not, in general, equivalent to $f$ vanishing everywhere, as it is possible for nonzero polynomials of degree greater than or equal to $q$ to vanish everywhere. \\

Finally, we define a piece of notation which will be useful later on.
For polynomials $\{f_i(x)\}_{i=1}^n \subset \F_q[x_1, \dots, x_d]$ and $\beta \in \Z_{\geq 0}^n$, we define
\[E_\beta(f_1(x), \dots, f_n(x)) := \prod_{i=1}^n (f_i(x))^{\beta_i}.\]
This is the polynomial obtained by substituting the tuple $(y_1 = f_1(x), \dots, y_n = f_n(x))$ into the monomial $y^\beta$. In particular, if each $f_i(x)$ is a monomial in $x$, $E_\beta(f_1(x), \dots, f_n(x))$ is also a monomial.

\subsection{The polynomial method}
\bigbreak

Dvir's proof \cite{Dvir} that Kakeya sets in $\F_q^n$ have cardinality bounded below by $c_nq^n$ uses what is now known as the polynomial method, and is centred around the idea that a polynomial of low degree cannot have too many zeros. If a Kakeya set $K$ is too small, one can obtain a nonzero polynomial of low degree that vanishes at every point in $K$. Evaluating this polynomial on the parametrized lines contained in $K$, we then obtain a single-variable polynomial of low degree that is zero everywhere, and thus is the zero polynomial. From this, it can be shown that our original polynomial must also vanish at too many points outside of $K$, which gives a contradiction.

We present two essential tools used in the polynomial method. The first is a statement about how small a set needs to be to guarantee that we can find a nonzero polynomial vanishing on it.
This is essentially a statement about solution sets of systems of homogeneous linear equations, as we can treat coefficients in polynomials vanishing at certain points as variables in a system of linear equations.

\begin{lemma}\cite[Lemma 2.3]{Guth}\label{PC}
    Let $S \subset \F_q^n$. If $|S| < \binom{D + n}{n}$, there exists a nonzero polynomial $f \in \F_q[x_1, \dots, x_n]$, with $\deg f \leq D$, such that $f(s) = 0$ for all $s \in S$.
\end{lemma}

The Schwartz-Zippel lemma is the second main tool in the polynomial method, and gives a bound on the number of zeros a polynomial can have over a Cartesian product of a set. 

\begin{lemma}\cite{Schwartz, Zippel}\label{SZ}
    Let $A \subset \F_q$. If $f \in \F_q[x_1, \dots, x_m]$ satisfies $\deg f \leq d$, then 
    \[|\{x \in A^m: f(x) = 0\}| \leq d|A|^{m-1}.\]
\end{lemma}

\subsection{Hasse derivatives and multiplicities} 

The method of multiplicities extends the polynomial method, and uses the idea that polynomials of not too high degree cannot vanish at too many points with high multiplicity. This allows us to work with polynomials of degree greater than the size of the field, and thus obtain sharper lower bounds on quantities of interest. The notion of the multiplicity with which a polynomial vanishes at a point uses Hasse derivatives, which we define below.

\begin{definition}
    Let $f(x) \in \F_q[x_1, \dots, x_m]$, and let $\beta \in \Z_{\geq 0}^m$. We define the $\beta$-th Hasse derivative of $f$, denoted by $f^{(\beta)}$, to be the coefficient of $y^\beta$ in $f(x+y)$. Then
    \[f(x+y) = \sum_{\beta} f^{(\beta)}(x) y^\beta.\]
\end{definition}

In our proof of Theorem \ref{mm-bound}, we use the explicit form for Hasse derivatives, which we also state here.

\begin{lemma}\cite[Lemma 2.9]{Trainor}\label{HD-form}
    Let $f(x) = \sum_{\alpha} c_\alpha x^\alpha \in \F_q[x_1, \dots, x_m]$, and let $\beta \in \Z_{\geq 0}^m$. Then \[f^{(\beta)}(x) = \sum_{\alpha} c_\alpha \binom{\alpha}{\beta} x^{\alpha-\beta}.\]
\end{lemma}

Now we define what it means for a polynomial over $\F_q^m$ to vanish at a point with a certain multiplicity. If we have a polynomial $p(x) \in \C[x]$, $p(x)$ vanishes at a point $a \in \C$ with multiplicity $k$ if $(x-a)^k \mid p(x)$ and $(x-a)^{k+1} \nmid p(x)$. Equivalently, $p(x)$ vanishes at $a \in \C$ with multiplicity $k$ if $p^{(k-1)}(a) = 0$ and $p^{(k)}(a) \neq 0$. This idea of using derivatives to measure the multiplicity with which a polynomial vanishes at a point can be extended to polynomials over finite fields and their Hasse derivatives.

\begin{definition}
    Let $a \in \F_q^m$, and let $A \subset \F_q^m$. We define the multiplicity of $f(x) \in \F_1[x_1, \dots, x_m]$ at $a$ to be \[\mult(f,a) := \sup\{M \in \Z_{\geq 0}: f^{(\beta)}(a) = 0 \text{ for all } \beta \in \Z_{\geq 0}^m \text{ with } |\beta| < M\}. \] If $\mult(f, a) \geq M$ for all $a \in A$, we say that $f$ vanishes on the set $A$ with multiplicity $M$.
\end{definition}

The following two lemmas provide lower bounds on the behaviour of multiplicities under taking Hasse derivatives and under function composition.

\begin{lemma}\cite[Lemma 2.4]{DKSS}\label{sum-derivative}
    Let $P \in \F_q[x_1, \dots, x_n]$ and $\beta \in \Z_{\geq 0}^n$. Then for $a \in \F_q^n$,
    \[\mult(P^{(\beta)}, a) \geq \mult(P, a) - |\beta|.\]
\end{lemma}

\begin{lemma}\cite[Proposition 2.5]{DKSS}\label{mult-composition}
    Let $P(x) \in \F_q[x_1, \dots, x_m]$, $Q(y) := (Q_1(y), \dots, Q_m(y)) \in (\F_q[y_1, \dots, y_j])^m$, and $a \in \F_q^j$. Then 
    \[\mult(P \circ Q, a) \geq \mult(P, Q(a)).\]
\end{lemma}

Finally, we present multiplicity-enhanced versions of the two main tools used in the polynomial method. The first is an extension of Lemma \ref{PC}, and states how small a set must be to guarantee existence of a polynomial of not too high degree vanishing on our set with a given multiplicity.

\begin{lemma}\cite[Proposition 3.1]{DKSS}\label{MM-PC}
    Let $A \subset \F_q^n$. If 
    \[\binom{M + n - 1}{n} \cdot |A| < \binom{D + n}{n},\]
    then there exists a nonzero polynomial $P \in \F_q[x_1, \dots, x_n]$ of degree at most $D$ vanishing on $A$ with multiplicity $M$.
\end{lemma}

The second main tool is a multiplicity-enhanced version of the Schwartz-Zippel lemma (Lemma \ref{SZ}). This gives the same upper bound on the number of zeros of a polynomial as the Schwartz-Zippel lemma, but here we count zeros with their multiplicity.

\begin{lemma}\cite[Lemma 2.7]{DKSS}\label{extended-SZ}
    Let $A \subset \F_q$. Let $P \in \F_q[x_1, \dots, x_n]$ be a nonzero polynomial of degree $d$. Then 
    \[\sum_{a \in A^n} \mult(P, a) \leq d|A|^{n-1}.\]
\end{lemma}

\subsection{Vanishing multiplicities in the context of $(n,d)$-BRK-type sets of degree $\ell$}

In this section, we adapt some of the methods of \cite{Trainor} specific to the constituent hypersurfaces in BRK-type sets to analogous methods for the $d$-dimensional subsets in $(n,d)$-BRK-type sets.

The first lemma is an analogue of \cite[Lemma 2.7]{Trainor}, which states that if a polynomial of sufficiently low degree vanishes on an $(n-1)$-dimensional subset of degree $\ell$ with high multiplicity, then it must be the zero polynomial. We observe that an identical result holds for our $(n-d)$-dimensional subsets, with an identical proof to \cite[Lemma 2.7]{Trainor}, using Lemma \ref{mult-composition} as a generalization of \cite[Lemma 2.3]{Trainor}.

\begin{lemma}\label{mult-inputsubset}
    Let $\ell \in \N$ with $1 \leq \ell < q$, and let $\{g_i\}_{i=1}^{n-d} \subset \F_q[t_1, \dots, t_d]$ be polynomials of degree $\ell$. For $\rho \in (\F_q \setminus \{0\})^{n-d}$ and $a \in \F_q^n$, let 
    \[C := \{a + (t, \rho_1 g_1(t), \dots, \rho_{n-k} g_{n-d}(t)): t \in \F_q^{d}\}.\]
    Let $k, D, M \in \N$ be such that 
    \begin{equation}
    \ell(D - \omega) < (M- \omega)q \label{M-D-ineq}
    \end{equation}
    for $\omega = k-1$.
    Suppose that $f$ is a non-zero polynomial of degree at most $D$ vanishing on $C$ with multiplicity $M$. Let $f^{(\beta)}$ denote the Hasse derivative of $f$ of order $\beta \in \Z_{\geq 0}^n$, and let 
    \[f_{\beta, \rho}(s) := f^{(\beta)}(a + (s, \rho_1g_1(s), \dots, \rho_{n-d}g_{n-d}(s))).\]
    Then for $|\beta| < k$, $f_{\beta, \rho}$ is the zero polynomial.
\end{lemma}

\begin{proof}\cite[Proof of Lemma 2.7]{Trainor}

    Let $\beta \in \Z_{\geq 0}^m$, with $|\beta| < k$. 
    Since $f$ has degree at most $D$, 
    $f^{(\beta)}$ has degree at most $D - |\beta|$, 
    so that $f_{\beta, \rho}$ has degree at most $\ell(D - |\beta|)$. 
    Additionally, 
    since $f$ vanishes on $C$ with multiplicity $M$, 
    by Lemma \ref{sum-derivative}, we have that $f^{(\beta)}$ vanishes on $C$ with multiplicity at least $M - |\beta|$. 
    Then Lemma \ref{mult-composition} implies that $f_{\beta, \rho}$ vanishes on $\F_q^d$ with multiplicity at least $M - |\beta|$.

    Notice that (\ref{M-D-ineq}) must also hold for all $\omega \leq k-1$, and so in particular must hold for $\omega = |\beta|$. It follows that
    \begin{equation*}
        \sum_{s \in \F_q^d} \mult(f_{\beta, \rho}, s) \geq (M - |\beta|) q^d > \ell(D -|\beta|) q^{d-1} \geq \deg{f_{\beta, \rho}} q^{d-1}.
    \end{equation*} 
    Applying the multiplicity-enhanced Schwartz-Zippel lemma (Lemma \ref{extended-SZ}), 
    we see that $f_{\beta, \rho}$ must be the zero polynomial.
\end{proof}

In our proof of Theorem \ref{mm-bound}, 
we encounter a family of polynomials $P_\gamma(\rho)$ that vanish everywhere and are ``almost'' the Hasse derivatives of a polynomial up to a multiplicity $k$. 
We would like to use this to conclude that our original polynomial vanishes everywhere with multiplicity $k$. 
To do this, we use the following lemma. 
In a sense, this is a higher-dimensional analogue of \cite[Lemma 3.1]{Trainor};
however, 
this differs because our dilation parameter only impacts the final $n-d$ coordinates in our subset.

\begin{lemma}\label{HD-reduction}
    Let $k, n \in \N$, let $\{b_\alpha\}_{\alpha \in \Z^n} \in \F_q$, and let $c \in \F_q^{n-d}$. 
    Suppose that for all $\beta \in \Z_{\geq 0}^n$ with $|\beta| < k$ and each  $t \in (\F_q \setminus \{0\})^d$, 
    the value 
    \[f_{\beta}(t) = \sum_{|\alpha| < k(q-1)} b_\alpha \binom{\alpha}{\beta} c^{\alpha' - \beta'}t^{\alpha' - \beta'}\]
    is zero, 
    where $\alpha' = (\alpha_{d+1}, \dots, \alpha_{n})$ (and similarly for $\beta'$.) 
    Then $b_\alpha = 0$ for all $\alpha$.
\end{lemma}
\begin{proof}
By a change of variables $s_i = c_i t_i$, we obtain polynomials
\[\tilde{f}_\beta(s) := \sum_{|\alpha| < k(q-1)} b_\alpha \binom{\alpha}{\beta} s^{\alpha' - \beta'},\]
which are zero for all $s \in (\F_q \setminus \{0\})^d$ whenever $|\beta| < k$.
Consider $\tilde{f}_0(s) = \sum_{|\alpha| < k(q-1)} b_\alpha s^\alpha$, and let $\beta' \in \Z_{\geq 0}^{n-d}$. The $\beta'$-th Hasse derivative of $\tilde{f}_0$ is
\begin{equation}
    \tilde{f}_0^{(\beta')}(s) = \sum_{|\alpha| < k(q-1)} b_\alpha \binom{\alpha'}{\beta'} s^{\alpha' - \beta'}.
\end{equation}
Taking $\beta = (0, \dots, 0, \beta')$, 
we have that 
\[\binom{\alpha}{\beta} = \prod_{i=1}^n \binom{\alpha_i}{\beta_i} = \Bigl(\prod_{i=1}^d \binom{\alpha_i}{0} \Bigr) \Bigl(\prod_{j=1}^{n-d} \binom{\alpha_{j+d}}{\beta'_j}\Bigr) = 1 \cdot \binom{\alpha'}{\beta'}.\]
It follows that $\tilde{f}_0^{(\beta')}(s) = \tilde{f}_{\beta}(s)$; 
observe also that $|\beta| = |\beta'| < k$. 
Then $\tilde{f}_0^{(\beta')}(s)$ is zero for all $s \in (\F_q \setminus \{0\})^d$. 
Since $\beta'$ was arbitrary,
$\tilde{f}_0$ is thus a polynomial of degree less than $k(q-1)$ vanishing on $(\F_q \setminus \{0\})^d$ with multiplicity $k$, 
and hence is identically zero by Lemma \ref{extended-SZ}.
\end{proof}

\subsection{The graded lexicographic order on $\F_q[x_1, \dots, x_d]$ and some associated lemmas}
\bigbreak

In our proofs, 
we isolate for specific coefficients of polynomials after evaluation on our constituent subsets. 
To do this, 
we want a notion of a `leading coefficient' of a polynomial in several variables, 
which requires a total order on the associated monic monomials. For this, we use the graded lexicographic order on monomials. This begins by sorting monic monomials by their total degree, which is necessary as the polynomials that determine our surfaces are only fixed in the highest homogeneous part. 
We then use the lexicographical ordering to distinguish between distinct monomials of each total degree; 
this allows us to always select a unique highest monomial in any polynomial. The following results are standard and can be seen from e.g. \cite[Ch. 2.2]{ideals-varieties-algorithms}. We include their proofs for completness.

\begin{definition}
    Let $\alpha, \beta \in \Z_{\geq 0}^d$. 
    We write $x^{\alpha} \prec x^{\beta}$ whenever $x^{\alpha} \prec x^{\beta}$ with respect to the graded lexicographical order. That is, $x^{\alpha} \prec x^{\beta}$ if $|\alpha| < |\beta|$, 
    or if $|\alpha| = |\beta|$ and $\alpha < \beta$ with respect to the lexicographical order. \\
    We extend this to $\F_q[x_1, \dots, x_m]$ as follows: we write $LM(f)$ for the leading monomial of $f(x) \in \F_q[x_1, \dots, x_d]$, and we say that $g(x) \prec f(x)$ if $LM(g) \prec LM(f)$.
\end{definition}

With this partial order in hand, 
we consider how it behaves under various operations on polynomials.

\begin{remark}
    It is immediate from the definition of $\prec$ that if $g(x) \prec f(x)$ and $h(x) \prec f(x)$,
    then $g(x) + h(x) \prec f(x)$.
\end{remark}

Next, 
we show that $\prec$ behaves as expected under monomial multiplication.

\begin{lemma}\label{mon-sum}
    If $x^\beta \prec x^\gamma$, 
    then $x^{\alpha + \beta} \prec x^{\alpha + \gamma}$.
\end{lemma}

\begin{proof}
If $|\beta| < |\gamma|$, 
the result follows since $|\alpha + \beta| = |\alpha| + |\beta| < |\alpha| + |\gamma| = |\alpha + \gamma|$.
Otherwise, 
$|\alpha + \beta| = |\alpha + \gamma|$, 
and there exists a minimal $j$ so that $\beta_j < \gamma_j$. 
Then the exponent of $x_j$ in $x^{\alpha + \beta}$ is $\alpha_j + \beta_j$, 
which is strictly less than $\alpha_j + \gamma_j$, 
so that $\alpha + \beta < \alpha + \gamma$ with respect to the lexicographical order.
\end{proof}

This can be extended to polynomial multiplication: 
monomial dominance with respect to $\prec$ is preserved under arbitrary products of polynomials, 
and no cancellation occurs at the highest term. 
We state this formally below.
\begin{lemma}\label{pol-expansion}
    Let $\{f_i(x)\}_{i=1}^N \subset \F_q[x_1, \dots, x_n]$.
    Then 
    \[LM\left(\prod_{i=1}^N f_i(x)\right) = \prod_{i=1}^N LM(f_i(x)).\]
\end{lemma}

\begin{proof}
We proceed by induction on $N$, 
and begin with the $N=2$ case. Suppose that $LM(f_1(x)) = a_{\alpha} x^{\alpha}$ and $LM(f_2(x)) = b_{\beta} x^{\beta}$. 
That is, we write $f_1(x) =a_{\alpha} x^{\alpha} + \sum_{\gamma \in A} a_{\gamma} x^{\gamma}$ and $f_2(x) = b_{\beta} x^{\beta} + \sum_{\mu\in B} b_{\mu} x^{\mu}$, 
where $a_{\alpha} \neq 0$,  $b_{\beta} \neq 0$, $x^{\gamma} \prec x^{\alpha}$ for all $\gamma \in A$, 
and $x^{\mu} \prec x^{\beta}$ for all $\mu \in B$. 
Then 
\begin{align*}
    f_1(x)f_2(x) &= \Bigl(a_\alpha x^\alpha + \sum_{\gamma \in A} a_{\gamma} x^{\gamma}\Bigr)\Bigl(b_\beta x^\beta + \sum_{\mu \in B} b_{\mu} x^{\mu}\Bigr) \\
    &= a_\alpha b_\beta x^{\alpha + \beta} + g(x),
\end{align*}
where \[g(x) = a_\alpha \sum_{\mu \in B} b_\mu x^{\alpha + \mu} + b_\beta \sum_{\gamma \in A} a_\gamma x^{\gamma + \beta} +\sum_{\substack{\gamma \in A \\ \mu \in B}} a_{\gamma} b_{\mu} x^{\gamma + \mu}.\]
Lemma \ref{mon-sum} implies that $x^{\alpha + \mu} \prec x^{\alpha + \beta}$ and $x^{\gamma + \mu} \prec x^{\gamma + \beta} \prec x^{\alpha + \beta}$ for all $\gamma \in A$ and all $\mu \in B$, 
so that $g(x) \prec x^{\alpha + \beta}$, as desired.

Now, suppose that the result holds for a product of $N$ polynomials, with $LM(f_i(x)) = c_{\alpha^i} x^{\alpha^i}$,
and consider $\prod_{i=1}^{N+1} f_i(x) = \Bigl(\prod_{i=1}^N f_i(x)\Bigr)f_{N+1}(x)$. 
By the inductive hypothesis, 
\[\prod_{i=1}^N f_i(x) = \Bigl(\prod_{i=1}^N c_{\alpha^i}\Bigr) x^{\sum_{i=1}^N \alpha^i} + G(x)\] 
for some $G(x) \in \F_q[x_1, \dots, x_n]$, 
where $G(x) \prec x^{\sum_{i=1}^N \alpha^i}$. 
The $N=2$ case then implies that 
\[\prod_{i=1}^{N+1} f_i(x) = \Biggl(\Bigl(\prod_{i=1}^N c_{\alpha^i}\Bigr) x^{\sum_{i=1}^N \alpha^i} + G(x)\Biggr)f_{N+1}(x) = \Bigl(\prod_{i=1}^{N+1} c_{\alpha^i}\Bigr) x^{\sum_{i=1}^{N+1} \alpha^i} + g(x)\]
for some $g(x)$, 
where $g(x) \prec x^{\sum_{i=1}^{N+1} \alpha^i}$. By induction, the result holds for all $N$.

\end{proof}

From this, 
it follows that in fact $\prec$ is preserved under polynomial multiplication.

\begin{corollary}\label{pol-prod-ineq}
    Suppose that $f(x) \prec g(x)$. 
    For any $h(x) \in \F_q[x_1, \dots, x_n]$, 
    $f(x)h(x) \prec g(x)h(x)$.
\end{corollary}

\begin{proof}
    Write $f(x) = a x^\alpha +F(x)$, 
    $g(x) = bx^\beta + G(x)$, 
    and $h(x) = c x^\gamma + H(x)$, 
    where $a,b$ and $c$ are nonzero, 
    $F(x) \prec x^\alpha$, $G(x) \prec x^\beta$, and $H(x) \prec x^\gamma$. Then by Lemma \ref{pol-expansion}, 
    we have 
    \[f(x)h(x) = acx^{\alpha + \gamma} + f_1(x),\]
    for some $f_1(x)$ with $f_1(x) \prec x^{\alpha + \gamma}$, 
    and 
    \[g(x)h(x) = bcx^{\beta + \gamma} + f_2(x)\]
    for some $f_2(x)$ with $f_2(x) \prec x^{\beta + \gamma}$. 
    By the definition of $f(x) \prec g(x)$, $x^\alpha \prec x^\beta$; 
    Lemma \ref{mon-sum} then implies that $x^{\alpha + \gamma} \prec x^{\beta + \gamma}$, and the result follows.
\end{proof}

Monomial dominance with respect to $\prec$ is also often preserved with respect to taking Hasse derivatives, 
in the case where each input term is itself some polynomial in $x$. 
We state and prove such a result for monomials below, and we use the notation defined in Section \ref{notation}.

\begin{corollary}\label{monomial-derivative-dominance}
    Let $\{\alpha^i\}_{i=1}^n \subset \Z_{\geq 0}^{n}$ and $\beta^1, \beta^2 \in \Z_{\geq 0}^n$, 
    and suppose that $E_{\beta^1}(x^{\alpha^1}, \dots, x^{\alpha^{n}}) \prec E_{\beta^2}(x^{\alpha^1}, \dots, x^{\alpha^{n}})$. 
    Suppose additionally that $\gamma \in \Z_{\geq 0}^n$ is such that $\beta^1 - \gamma, \beta^2 - \gamma \in \Z_{\geq 0}^n$. 
    Then 
    \[E_{\beta^1 - \gamma}(x^{\alpha^1}, \dots, x^{\alpha^{n}}) \prec E_{\beta^1 - \gamma}(x^{\alpha^1}, \dots, x^{\alpha^{n}}).\]
\end{corollary}

\begin{proof}
    This is a consequence of the fact that 
    \[ E_{\beta - \gamma}(x^{\alpha^1}, \dots, x^{\alpha^{n}})E_{\gamma}(x^{\alpha^1}, \dots, x^{\alpha^{n}}) = E_{\beta}(x^{\alpha^1}, \dots, x^{\alpha^{n}})\]
    for any $\beta \in \Z_{\geq 0}^n$ and any $\gamma \in \Z_{\geq 0}^n$ with $\beta - \gamma \in \Z_{\geq 0}^n$,
    and we prove the contrapositive. 

    Since both polynomials we compare are monic monomials in $x$, and $\prec$ gives a total order on such monomials, 
    they are always either equal or comparable via $\prec$. 
    If $E_{\beta^1 - \gamma}(x^{\alpha^1}, \dots, x^{\alpha^{n}}) = E_{\beta^2 - \gamma}(x^{\alpha^1}, \dots, x^{\alpha^{n}})$,
    then 
    \[E_{\beta^1}(x^{\alpha^1}, \dots, x^{\alpha^{n}}) = E_{\beta^2}(x^{\alpha^1}, \dots, x^{\alpha^{n}}).\]
    If $E_{\beta^2 - \gamma}(x^{\alpha^1}, \dots, x^{\alpha^{n}}) \prec E_{\beta^1 - \gamma}(x^{\alpha^1}, \dots, x^{\alpha^{n}})$, Corollary \ref{pol-prod-ineq} implies that 
    \[E_{\beta^2}(x^{\alpha^1}, \dots, x^{\alpha^{n}}) \prec E_{\beta^1 }(x^{\alpha^1}, \dots, x^{\alpha^{n}}).\]
    In both cases, we do not have \[E_{\beta^1}(x^{\alpha^1}, \dots, x^{\alpha^{n}}) \prec E_{\beta^2}(x^{\alpha^1}, \dots, x^{\alpha^{n}}).\]
\end{proof}

\section{Proof of Theorem \ref{pm-bound}}\label{pm-proof}

Let $S$ be an $(n,d)$-BRK-type set, and let $t = (t_1, \dots, t_d)$. Let $\{h_i(t)\}_{i=1}^{n-d}$ and $\Gamma_\rho$ be as in Definition \ref{BRK-type-general}. Suppose, 
towards contradiction, 
that $|S| < \binom{\left\lfloor(q-1)/\ell\right\rfloor + n}{n}$. 
By Lemma \ref{PC}, 
there exists a nonzero polynomial $f \in \F_q[x_1, \dots, x_n]$ with $\deg f \leq \left\lfloor \frac{q-1}{\ell}\right\rfloor$ vanishing on $S$. 
We write $f(x) = \sum_{\beta \in B} b_\beta x^\beta$, where $B \subset \Z_{\geq 0}^n$ is the set of multi-indices $\beta$ with $b_\beta \neq 0$.

We begin by considering the polynomials $\{h_i(t)\}_{i=1}^{n-d}$. 
For each $i$, we may write $h_i(t) = c_{\alpha^i} t^{\alpha^i} + H_i(t)$, 
where $t^{\alpha^i} \succ H_i(t)$. 
We consider $LM(f(t, t^{\alpha^1}, \dots, t^{\alpha^{n-d}}))$, 
the leading monomial obtained by evaluating $f$ at $(t, t^{\alpha^1}, \dots, t^{\alpha^{n-d}})$ 
and considering this as a polynomial in $t$.
That is, we write 
\[f(x) = \sum_{j=1}^N b_{\beta^j} x^{\beta^j} + F(x)\]
for some $N \in \N$, where the $\beta^j$ satisfy 

\[E_{\beta^i}(t, t^{\alpha^1}, \dots, t^{\alpha^{n-d}}) = E_{\beta^j}(t, t^{\alpha^1}, \dots, t^{\alpha^{n-d}}) \]

for all $1 \leq i, j \leq N$, and
\[E_{\beta^1}(t, t^{\alpha^1}, \dots, t^{\alpha^{n-d}}) \succ F(t, t^{\alpha^1}, \dots, t^{\alpha^{n-d}}) \]

Now, let $\rho \in (\F_q \setminus \{0\})^{n-d}$. Since $f$ vanishes on $S$, we must have 
\[f_\rho(t) := f(\Gamma_\rho(t)) = 0\]
for all $t \in \F_q^d$. 
Since the input in each coordinate is a polynomial of total degree at most $\ell$, 
it follows that $\deg f_\rho \leq \deg f \cdot \ell \leq q-1$, 
and thus that $f_\rho(t)$ is the zero polynomial.

Applying Lemma \ref{pol-expansion} to $f_\rho(t)$, we obtain
\begin{align*}
    f_\rho(t) &= \sum_{\beta \in B} b_\beta (\Gamma_\rho(t))^\beta \\
    &= \sum_{\beta \in B} b_\beta\Bigl(E_\beta(t, \rho_1 c_{\alpha^1}t^{\alpha^1}, \dots, \rho_{n-d} c_{\alpha^{n-d}} t^{\alpha^{n-d}}) + F_{\rho, \beta}(t)\Bigr)\\
    &= \sum_{\beta \in B} b_\beta\Bigl(\Bigl(\prod_{i=1}^{n-d} \rho_i^{\beta_{i+d}}c_{\alpha^i}^{\beta_{i+d}}\Bigr) E_\beta(t, t^{\alpha^1}, \dots, t^{\alpha^{n-d}}) + F_{\rho, \beta}(t)\Bigr),
\end{align*}
where $F_{\rho, \beta}(t) \prec E_{\beta; \alpha^1, \dots, \alpha^{n-d}}(t) $ for each $\beta$. 
By our choice of $\{\beta^j\}_{j=1}^N$, 
the leading coefficient of $f_\rho(t)$ is thus given by 
\[P(\rho) = \sum_{j=1}^N b_{\beta^j} \rho^{(\beta^j)'}c^{(\beta^j)'},\]
where $(\beta^j)'$ are defined as in Section \ref{notation}.
This is zero, since $f_\rho$ is the zero polynomial.
But $\rho$ was arbitrary: 
$P$ is thus a nonzero polynomial of total degree at most $\frac{q-1}{\ell}$ that vanishes at every $\rho \in (\F_q \setminus \{0\})^{n-d}$. 
This is the desired contradiction, as Lemma \ref{SZ} asserts that such a polynomial has no more than $\frac{(q-1)^{n-d}}{\ell} < (q-1)^{n-d}$ zeros in $(\F_q \setminus \{0\})^{n-d}$. \qed

\section{Proof of Theorem \ref{mm-bound}}\label{mm-proof}

We obtain our lower bound on the cardinality of $(n,d)$-BRK-type sets of degree $\ell$ from a sequence of lower bounds.

\begin{lemma}\label{mm-claim}
    Let $S$ be an $(n,d)$-BRK-type set of degree $\ell$, and let $k \in \N$ be a multiple of $q$. Let  $D = k(q-1) - 1$ and $M = (\ell + 1)k - 2\ell k/q$. Then 
    \[\binom{M + n - 1}{n} |S| \geq \binom{D + n}{n}.\]
\end{lemma}

\begin{proof}

Suppose, again towards contradiction, that there exists $k \in \N$, a multiple of $q$, so that for $D = k(q-1) - 1$ and $M = (\ell + 1)k - 2\ell k/q$, \[\binom{M + n - 1}{n} |S| < \binom{D + n}{n}.\]
By Lemma \ref{MM-PC}, there exists a nonzero polynomial $f \in \F_q[x_1, \dots, x_n]$ with $\deg f \leq D$ vanishing on $S$ with multiplicity $M$. 
We again write $f(x) = \sum_{\beta \in B} b_\beta x^\beta$, where $B \subset \Z_{\geq 0}^n$ is the set of multi-indices $\beta$ with $b_\beta \neq 0$.

We again consider the polynomials $\{h_i(t)\}_{i=1}^{n-d}$. 
For each $i$, we write $h_i(t) = c_{\alpha^i} t^{\alpha^i} + H_i(t)$, where $t^{\alpha^i} \succ H_i(t)$. Define $c = (c_{\alpha^1}, \dots, c_{\alpha^{n-d}}) \in \F_q^{n-d}$.
We again consider $LM(f(t, t^{\alpha^1}, \dots, t^{\alpha^{n-d}}))$, 
the leading monomial obtained by evaluating $f$ at $(t, t^{\alpha^1}, \dots, t^{\alpha^{n-d}})$ and considering this as a polynomial in $t$.
As in the proof of Theorem \ref{pm-bound}, we write
\[f(x) = \sum_{j=1}^N b_{\beta^j} x^{\beta^j} + F(x)\]
where the $\beta^j$ satisfy 
\[E_{\beta^i}(t, t^{\alpha^1}, \dots, t^{\alpha^{n-d}}) = E_{\beta^j}(t, t^{\alpha^1}, \dots, t^{\alpha^{n-d}})\]
for all $1 \leq i, j \leq N$, and
\[E_{\beta^1}(t, t^{\alpha^1}, \dots, t^{\alpha^{n-d}}) \succ F(t, t^{\alpha^1}, \dots, t^{\alpha^{n-d}}).\]
Let $\rho \in (\F_q \setminus \{0\})^{n-d}$. Since $f$ vanishes on $S$ with multiplicity $M$, 
we must have 
\[f_{\gamma, \rho}(t) := f^{(\gamma)}(\Gamma_\rho(t)) = 0\]
for all $t \in \F_q^d$ and all $|\gamma| < M$. 
Since $\ell, D, M$, and $\omega=k-1$ satisfy the inequality \eqref{M-D-ineq}, 
Lemma \ref{mult-inputsubset} states that for $|\gamma| < k$, 
$f_{\gamma, \rho}$ is the zero polynomial. 

Fix some $|\gamma| < k$. 
Applying Lemmas \ref{HD-form} and \ref{pol-expansion} to $f_{\gamma, \rho}(t)$, 
we obtain
\begin{align*}
    f_{\gamma, \rho}(t)
    &= \sum_{\substack{\beta \in B \\ \beta - \gamma \in \Z_{\geq 0}^n}} b_\beta \binom{\beta}{\gamma}\Bigl(\Bigl(\prod_{i=1}^{n-d} \rho_i^{\beta_{i+d} - \gamma_{i+d}}c_{\alpha^i}^{\beta_{i+d} - \gamma_{i+d}}\Bigr) E_{\beta - \gamma}(t, t^{\alpha^1}, \dots, t^{\alpha^{n-d}}) + F_{\rho, \beta, \gamma}(t)\Bigr)
\end{align*}
where $F_{\rho, \beta, \gamma} \prec E_{\beta - \gamma}(t, t^{\alpha^1}, \dots, t^{\alpha^{n-k}})$ for all $\beta$. 
To find the leading coefficient of $f_{\gamma, \rho}(t)$, 
by Lemma \ref{pol-expansion} it thus suffices to compare $E_{\beta - \gamma}(t, t^{\alpha^1}, \dots, t^{\alpha^{n-k}})$ for all $\beta \in B$ with $\beta - \gamma \in \Z_{\geq 0}^n$. 
If $\beta^j - \gamma \in \Z_{\geq 0}^n$ for any $j$ 
(without loss of generality, we assume $j=1$),
Corollary \ref{monomial-derivative-dominance} implies that the coefficient corresponding to the monomial $(t, t^{\alpha^1}, \dots, t^{\alpha^{n-k}})^{\beta^1 - \gamma}$ is 
\[P_{\gamma}(\rho):= \sum_{j=1}^N b_{\beta^j} \binom{\beta^j}{\gamma} \rho^{(\beta^j)' - \gamma'} c^{(\beta^j)' - \gamma'}  = \sum_{\substack{1 \leq j \leq N \\ \beta^j - \gamma \in \Z_{\geq 0}^n}} b_{\beta^j} \binom{\beta^j}{\gamma} \rho^{(\beta^j)' - \gamma'} c^{(\beta^j)' - \gamma'}\]
and is zero, 
where we use that the $\beta^j$ not included in the second sum satisfy $\binom{\beta^j}{\gamma} = 0$ as in Lemma \ref{HD-form}. 
Otherwise, $\binom{\beta^j}{\gamma} = 0$ for all $1 \leq j \leq N$, and hence in this case we also have that $P_{\gamma}(\rho) = 0$. 

Since $|\gamma| < k$ was arbitrary, Lemma \ref{HD-reduction} implies that each $b_{\beta^j}$ is zero. This contradicts our initial assumption that each $b_\beta$ is nonzero.

\end{proof}

Hence for each $k \in \N$, 
$k$ a multiple of $q$, with $D = k(q-1)-1$ and $M = (\ell + 1)k - 2\ell k/q$,
we have 
\[|S| \geq \frac{\binom{D + n}{n}}{\binom{M + n - 1}{n}} = \frac{\prod_{i=1}^n (k(q-1) - 1 + i)}{\prod_{i=1}^n ((\ell + 1)k - 2\ell k/q + i - 1)}.\]
This holds for arbitrarily large $k$; 
taking the supremum over all such $k$, 
we obtain 
\[|S| \geq \frac{(q-1)^n}{((\ell + 1) - 2\ell/q)^n},\]
as desired. \qed

\section{Acknowledgements}

This work was partially supported by an Undergraduate Student Research Award from the Natural Sciences and Engineering Research Council of Canada (NSERC). I would like to thank Malabika Pramanik for her encouragement and many helpful suggestions.

\Addresses 
}

\end{document}